\documentclass{amsart}

\usepackage{amssymb,amsmath,amsthm,latexsym,amscd,amsfonts,mathrsfs,mathtools}
\usepackage{epsfig}
\usepackage{graphicx}
\usepackage{bbold}
\usepackage{color}
\usepackage[all,cmtip]{xy}

\newtheorem{Theorem}{Theorem}
\newtheorem{Lemma}[Theorem]{Lemma}
\newtheorem{Corollary}[Theorem]{Corollary}
\newtheorem{Proposition}[Theorem]{Proposition}
\newtheorem{Remark}[Theorem]{Remark}
\newtheorem{Example}[Theorem]{Example}

\newtheorem{Question}[Theorem]{Question}

\newcommand{\fkm}{\mathfrak m}
\newcommand{\Rees}{\mathcal R}

\DeclareMathOperator{\Sym}{Sym}
\DeclareMathOperator{\Fitt}{Fitt}

\renewcommand{\tilde}{\widetilde}
\renewcommand{\bar}{\overline}

\makeatletter
\@namedef{subjclassname@2020}{\textup{2020} Mathematics Subject Classification}
\makeatother

\makeatletter
 
 \@addtoreset{equation}{section}
 \makeatother

\title{
Note on length and multiplicity
of modules over two-dimensional regular local rings
}

\author{Futoshi Hayasaka}
\address{Department of Environmental and Mathematical Sciences, Okayama University, 3-1-1 Tsushimanaka, Kita-ku, Okayama, 700-8530, JAPAN}
\email{hayasaka@okayama-u.ac.jp}
\author{Vijay Kodiyalam}
\address{The Institute of Mathematical Sciences, Chennai, India and Homi Bhabha National Institute, Mumbai, India}
\email{vijay@imsc.res.in}

\begin{document}

\begin{abstract}
We give lower and upper bounds on the Buchsbaum-Rim multiplicity of finitely generated torsion-free modules over two-dimensional regular local rings, and conditions for them to attain the bounds. As consequences, we have formulae on the multiplicity of integrally closed modules. 
\end{abstract}

\maketitle

\markright{LENGTH AND MULTIPLICITY OF MODULES}


\section{Introduction} 

The theory of integrally closed ideals in two-dimensional regular local rings, which was founded by Zariski in \cite{Zrs1938, ZrsSml1960}, has been deeply studied and extended in some directions by several authors. 
This theory was extended in \cite{Kdy1995} 
to finitely generated torsion-free integrally closed modules over two-dimensional regular local rings, and 
an analogue of Zariski's product theorem was obtained. 
Some other analogue results have been obtained in \cite{KtzKdy1997, Mhn1997, KdyMhn2015, Sbt2019}. 

In this note, we give several inequalities on  the length and multiplicity of modules over two-dimensional regular local rings. Among other things, we give lower and upper bounds on the Buchsbaum-Rim multiplicity of modules in terms of the ideal of maximal minors and its adjoint. Moreover, giving conditions for them to  attain  the bounds, we have a formula for integrally closed modules, which can be viewed as an analogue of the  corresponding well-known fact for integrally closed ideals. 

In order to state the results, let $M$ be a finitely generated torsion-free module of rank $r$ 
over a two-dimensional regular local ring $R$, 
and let $F$ be the double $R$-dual $M^{\ast \ast}$ of $M$. The module $F$ is $R$-free containing $M$ canonically, and the quotient $F/M$ is of finite length. 
Let $I(M)$ be the ideal of maximal minors of a matrix  whose columns generate the module $M$ with respect to some basis of $F$. Let $e(I)$ denote the multiplicity of $I$ and $e(M)$ denote the Buchsbaum-Rim multiplicity of $M$. Then our results can be summarized as follows: 

\begin{Theorem}\label{main}
Let $(R, \fkm)$ be a two-dimensional regular local ring with infinite residue field, and 
let $M$ be a finitely generated torsion-free $R$-module with ideal of maximal minors
$I=I(M)$. 
Then we have the inequalities:  
$$e(I)-\lambda(R/adj(I)) \leq e(M) \leq \lambda(F/M)+\lambda(R/adj(I))$$
where $adj(I)$ denotes the adjoint of $I$. Moreover, 
\begin{enumerate}
    \item The equality 
    $e(M) = \lambda (F/M)+\lambda (R/adj(I))$  holds if and only if 
    the module $M$ is integrally closed. 
    \item The equality $e(M)=e(I)-\lambda (R/adj(I))$ holds 
    if and only if the ideal $I$ is integrally closed with $  \lambda (R/I)=e(M)$.
\end{enumerate}
\end{Theorem}

As a consequence, we obtain a formula: 
$$e(M) = \lambda (F/M)+\lambda (R/adj(I))$$ 
for integrally closed modules $M$, which can be viewed as an analogue of the well-known fact: 
$$
e(I )=  \lambda (R/I )+ \lambda (R/adj(I ))
$$
for integrally closed $\fkm$-primary ideals $I $ in $R$. 
As an advantage of this extension, we can readily get 
the following formula - Corollary \ref{kdymhn} - for integrally closed modules:  
$$ e(I)-  e(M) =\lambda (R/I)-  \lambda (F/M)$$
which was discovered 
in \cite{KdyMhn2015}. 
Our proof is quite different from theirs which is based on the Hoskin-Deligne formula. Thus, we give a new approach to this interesting formula. 
Furthermore, we prove the new inequality $e(I)-e(M)\geq \lambda(R/I)-\lambda(F/M)$ for any finitely generated torsion-free $R$-module $M$.

As another advantage of the extension, we prove that for any given integrally closed $\mathfrak m$-primary ideal $I$ of order $r$, 
there is a one to one correspondence - Theorem \ref{1to1} - between the isomorphism classes of integrally closed modules $M$ of rank $r$ with $I(M)=I$ and those of contracted modules $K$ of rank $r$ with $I(K)=I$ and $I_{r-1}(K)=adj(I)$. Here, $I_{r-1}(K )$ denotes the ideal generated by $(r-1) \times (r-1)$-minors 
of a matrix whose columns generate $K $. 

In \S 2, we will fix our notation and recall some basic facts we will use in this note. 
In \S 3, we will prove our main result - Theorem \ref{main} - and give some applications. 
Finally, in \S 4, we make some remarks illustrated with a concrete example.  

\section{Preliminaries} 

Throughout this note, $(R,\fkm)$ will be a two-dimensional regular local ring with infinite residue field $R/\fkm$, and $M$
will be a non-free, finitely generated, torsion-free $R$-module of rank $r$. Let $F$ be the double $R$-dual $M^{\ast \ast}$ of $M$. Then $F$ is an $R$-free module of rank $r$ containing $M$ canonically with the quotient $F/M$ of finite non-zero length
 - see Proposition 2.1 of \cite{Kdy1995}. 
We regard $M$ as a submodule of $F$ generated by the columns of a suitable matrix. 
To be precise, let $F=RT_1+\dots +RT_r$ with  basis $T_1, \dots , T_r$. Suppose that $M=(f_1, \dots , f_n)$ is generated by $f_1, \dots , f_n$. 
Define the associated matrix, denoted by $\widetilde{M}$, as the $r \times n$ matrix $(a_{ij})$ where $f_j=a_{1j}T_1+\dots +a_{rj}T_r$. 
We then identify $M$ as the submodule of $F$ generated by the columns of $\widetilde{M}$. 
Let $I_k(M)$ denote the ideal generated by the $k \times k$-minors of 
$\widetilde{M}$.
This ideal is the $(r-k)$th Fitting ideal of $F/M$, so
it is independent of the choice of $\widetilde{M}$, thereby justifying the notation. 
Let $I(M)=I_r(M)$ be the ideal generated by maximal minors of $\widetilde{M}$. 
By $\lambda(M )$ and $\mu(M )$, we mean the length and the number of minimal generators of  $M $ respectively. 

First, we recall some basic facts on the Buchsbaum-Rim multiplicities, reductions and the integral closure for modules 
over two-dimensional regular local rings. 
The {\it Buchsbaum-Rim multiplicity} of $M$ 
defined in \cite{BchRim1964} 
is a positive integer
$$e(M)=\lim_{p \to \infty}(r+1)! \frac{  \lambda (\Sym^p_R(F)/M^p)}{p^{r+1}} $$
where $M^p=\mathrm{Im}(\Sym_R^p(M) \to \Sym_R^p(F))$ is the image of the natural homomorphism.
Note that $e(M)$ is independent of a choice of the presenting matrix of $F/M$  - see Theorem 3.3 of \cite{BchRim1964} - thereby justifying the notation. 

Let $\Rees(M)=\mathrm{Im}(\Sym_R(M) \to \Sym_R(F))=\oplus_{p \geq 0} M^p$. 
Then $\Rees(M)$ is a graded subalgebra of the polynomial ring $\Sym_R(F)=R[T_1,\cdots,T_r]$, and  
the homogeneous component of degree $p$ is denoted by $M^p=\mathrm{Im}(\Sym^p_R(M) \to \Sym^p_R(F))$. 
A submodule $N$ of $M$ is said to be a {\it reduction} of $M$
if the ring extension $\Rees(N) \subset \Rees(M)$ is integral or equivalently if 
the equality $M^{p+1}=NM^p$ holds in $\Sym_R^{p+1}(F)$ for some $p \geq 0$. 

A reduction $N$ of $M$ is said to be 
{\it minimal} if $N$ itself has no proper reduction. A minimal reduction $N$ of $M$ 
always exists, and every minimal 
generating set of $N$ can be extended to a minimal generating set of $M$. 
The proof is the same as in the case of ideals - see Theorem 8.3.6 and 8.3.3 of \cite{HnkSwn2006}. 
Similarly, since the residue field $R/\fkm$ is infinite, 
we have the following result of Rees - see Lemma 2.2 of \cite{Res1987}. 

\begin{Proposition} 
Let $M$ be a finitely generated torsion-free $R$-module of rank $r$. Then the inequality
$\mu (N) \geq r+1$ 
holds for any reduction $N$ of $M$. Moreover, equality holds if and only if $N$ is a minimal reduction of $M$. 
\end{Proposition}
 
A submodule $N$ of $F$ is said to be a {\it parameter module} in $F$ if the following three conditions are satisfied: 
(i) $  \lambda (F/N)<\infty$, (ii) $N \subset \fkm F$ and (iii) $\mu(N)=r+1$.
Therefore, for $M$ without free direct summands, any minimal reduction $N$ of $M$ 
is a parameter module in $F$. 
We will need the following result on the Buchsbaum-Rim multiplicity which follows from Proposition 3.8 of \cite{Kdy1995} and Corollary 4.5 of \cite{BchRim1964}. See also Theorem 1.3 of \cite{HysHry2010}.

\begin{Theorem}\label{br-formula}
Let $M$ be a finitely generated torsion-free $R$-module with $F=M^{\ast \ast}$. Then the equalities 
$$e(M)=e(N) =\lambda (F/N)= \lambda (R/I(N))$$ 
hold for any minimal reduction $N$ of $M$.
\end{Theorem}

The {\it integral closure} $\bar M$ of $M$ 
defined in \cite{Res1987} 
is a submodule of $F$ containing $M$, and it can be 
expressed as $$\bar M=\{ f \in F \mid I(M) \ \text{is a reduction} \text{~of} \  I(M+Rf) \} $$
- see Theorem 3.2 of \cite{Kdy1995}. Therefore, $N$ is a reduction of $M$ if and only if $I(N)$ is a reduction of $I(M)$. 
Since $R$ is a two-dimensional regular local ring,
we have the following useful formula - see Theorem 5.4 of \cite{Kdy1995} - which (when applied to $M=I \oplus J$ with $I,J$ integrally closed) can be viewed as an analogue of the classical Zariski's product theorem.  
\begin{Theorem}\label{kdy}
Let $M$ be a finitely generated torsion-free $R$-module with ideal of minors $I(M)$. 
Then the equality $$\bar{I(M)}=I(\bar M)$$ holds. In particular, the ideal $I(M)$ is integrally closed if $M$ is integrally closed. 
\end{Theorem}

Next, we describe the length of $M/N$ for a minimal reduction $N$
of $M$. Assume that $M$ has no free direct summands, equivalently, $M \subset \fkm F$. 
Let $N=(f_1, \dots , f_{r+1})$ be a minimal reduction of $M$. 
Extending the generating set to a minimal generating set of $M$, we can write $M=(f_1, \dots, 
f_{r+1}, \dots , f_n)$ where $n=\mu(M)$. 
Consider a minimal free resolution:
$$  \begin{CD}
     0 @>>> R^{n-r}  @>{A}>>  R^{n}
     @>{[f_1 \cdots f_n]}>>  M  @>>>  0, 
  \end{CD}
$$
where $A$ is a presenting matrix of $M$. 
Let $B$ be the submatrix of $A$ obtained by deleting the
first $(r+1)$ rows.

\begin{Lemma}\label{keylem}
With notation as above, we have $\lambda(M/N)=\lambda(R/I_{n-r-1}(B))$ 
where $I_{n-r-1}(B)$ is the ideal generated by maximal minors of $B$. 
\qed
\end{Lemma}

This lemma follows from the proof of Proposition 4.1 of \cite{KdyMhn2015}. 
Also, the particular case can be found in 
Lemma 4.7 of \cite{Hys2022}.

\section{Length and multiplicity} 

In order to prove Theorem \ref{main}, we investigate the following two differences: 
$$e(M)-\lambda(F/M) \ \text{and} \  e(I(M))-e(M). $$ 
First, we give the lower bounds in the following two propositions.   

\begin{Proposition}\label{lower1}
Let $M$ be a finitely generated torsion-free $R$-module of rank $r$ without free direct summands.  
Then $e(M)-\lambda(F/M) \geq \lambda(R/\Fitt_{r+1}(M))$. 
\end{Proposition}

\begin{proof}
Let $N$ be a minimal reduction of $M$. Extending a minimal generating set of $N$
to a set of minimal generators of $M$, we get a presenting matrix $A$ of $M$ and its
submatrix $B$ as in Lemma \ref{keylem}. Then we have 
\begin{align*}
e(M)-\lambda(F/M) &=\lambda(F/N)-\lambda(F/M) \quad \text{by Theorem $\ref{br-formula}$}\\
&=\lambda(M/N)\\
&=\lambda(R/I_{n-r-1}(B)) \quad \text{by Lemma $\ref{keylem}$}\\
&\geq \lambda(R/I_{n-r-1}(A))\\
&=\lambda(R/\Fitt_{r+1}(M))
\end{align*}
as desired. 
\end{proof}

\begin{Proposition}\label{lower2}
Let $M$ be a finitely generated torsion-free $R$-module of rank $r$ with 
$I=I(M)$. 
Then $e(I)-e(M) \geq \lambda(R/I_{r-1}(M))$. 
\end{Proposition}

In the course of the proof we need to appeal to the following 
 - see Proposition 6 of \cite{HysKdy2023}. 

\begin{Proposition}\label{rk3prop}
For a non-free, finitely-generated torsion-free $R$-module $M$ of rank $r$, there exists
a minimal reduction $N$ of $M$ such that if the transpose 
$\widetilde{N}^T$ of a matrix $\widetilde{N}$ whose columns generate $N$
resolves 
the ideal $I(N)=(a_1, a_2, \dots , a_{r+1})$, then $(a_1, a_2)$ is a minimal reduction of $I(N)$.
\qed 
\end{Proposition}

\begin{proof}[Proof of Proposition $\ref{lower2}$]
We may assume that $M$ has no free direct summands. 
Choose a minimal reduction $N$ and a matrix $\tilde{N}$ whose columns generate $N$ as in Proposition \ref{rk3prop}, i.e., 
if 
$\widetilde{N}^T$ 
resolves the ideal $I(N)=(a_1, a_2, \dots , a_{r+1})$, 
then $(a_1, a_2)$ is a minimal reduction of $I(N)$. 
Note that $(a_1, a_2)$ is also a minimal reduction of $I$ since $I(N)$ is a reduction of $I$.
We will apply Lemma \ref{keylem} when $M=I(N)$ and $N=(a_1, a_2)$. 
Then the matrix 
$\widetilde{N}^T$ 
can be chosen as the presenting matrix $A$ as in Lemma \ref{keylem}, and 
$B$ is the submatrix of 
$\widetilde{N}^T$ 
obtained by deleting the first two rows. 

When this is the case, by Lemma \ref{keylem}, $\lambda(I(N)/(a_1, a_2))=
\lambda(R/I_{r-1}(B))$. 
Therefore, 
\begin{align*}
    e(I)-e(M) &=\lambda(R/(a_1, a_2))-\lambda(R/I(N)) \quad \text{by Theorem $\ref{br-formula}$} \\
    &=\lambda(I(N)/(a_1, a_2))\\
    &=\lambda(R/I_{r-1}(B))\\
    &\geq   \lambda (R/I_{r-1}(N)) \\
    &\geq   \lambda (R/I_{r-1}(M))
\end{align*}
Thus, we have the desired inequality. 
\end{proof}

Next, we give the upper bounds of the differences. 
For this, we will need to recall a few facts on adjoint ideals in two-dimensional regular local rings. 

The {\it adjoint} of an ideal $I$ in $R$ is defined by Lipman in \cite{Lpm1994} as  
$$adj(I)=\bigcap_{V} \{ a \in K \mid a J_{V/R} \subset IV \}, $$
where $K$ is the quotient field of $R$ and 
the intersection is taken over all divisorial valuation rings $V$ with respect to $R$, and 
$J_{V/R}$ denotes the Jacobian ideal of $V$ over $R$. 
Then $adj(I)$ is an integrally closed ideal in $R$ and satisfies 
\begin{equation}\label{adjprop}
I \subset \bar{I } \subset adj(I )=adj(\bar{I })
\end{equation} 
- see Lemma 18.1.2 of \cite{HnkSwn2006} for instance. 

Huneke and Swanson also proved in \cite{HnkSwn2006} that for an integrally closed $\fkm$-primary ideal $I$ in a two-dimensional regular local ring $R$, the adjoint $adj(I)$ can be obtained from the presenting matrix $A$ of $I$. 
To state the result precisely, let $I =(a_1, a_2, \dots , a_n )$ be an integrally closed $\fkm$-primary ideal with $n=\mu(I )$.
Let $A$ be a presenting matrix of $I$ in the following exact sequence: 
\begin{equation}\label{pres}
  \begin{CD}
     0 @>>> R^{n-1}  @>{A}>>  R^{n}
     @>{[a_1 \cdots a_n]}>>  I  @>>>  0.
  \end{CD}
\end{equation}
Then $I=I_{n-1}(A)$ the ideal of maximal minors of $A$ by the Hilbert-Burch theorem. Theorem 18.5.1 of \cite{HnkSwn2006} is the following.

\begin{Theorem}
Let $I =(a_1, a_2, \dots , a_n )$ be an integrally closed $\fkm$-primary ideal with $n=\mu(I )$.
Let $A$ be a presenting matrix of $I$ in the above exact sequence (\ref{pres}). 
Then we have 
$$adj(I )=I_{n-2}(A). $$
Moreover, if the first two generators $a_1, a_2$ form a minimal reduction
of $I $, then  
$$adj(I )=I_{n-2}(B)$$ 
where $B$ is the submatrix of $A$ obtained by deleting the first two rows. \qed
\end{Theorem}
\color{black}

This can be extended to integrally closed modules. 
Shibata proved in \cite{Sbt2019} that for an integrally closed module $M$ over a two-dimensional regular local ring $R$, 
the adjoint $adj(I(M))$ of the ideal of minors $I(M)$ can be obtained from the presenting matrix of $M$ - see Theorem 3.1 of \cite{Sbt2019}. 

\begin{Theorem}\label{adj}
Let $M=( f_1, \dots , f_n )$ be a non-free, integrally closed $R$-module of rank $r$ with 
$n=\mu(M)$. Suppose that the first $r+1$ generators form 
a minimal reduction $N=(f_1, \dots , f_{r+1} )$ of $M$. 
Let $A$ be a presenting matrix of $M$ in the following exact sequence: 
\begin{equation*}
  \begin{CD}
     0 @>>> R^{n-r}  @>{A}>>  R^{n}
     @>{[f_1 \cdots f_n]}>> M  @>>>  0.
  \end{CD}
\end{equation*}
Let $B$ be the submatrix of $A$ obtained by deleting the first $(r+1)$ rows as in Lemma \ref{keylem}.  Then we have the equalities
\begin{equation*} 
    adj(I(M))=I_{n-r-1}(A)=I_{n-r-1}(B). \qed
\end{equation*}
\end{Theorem}
See also Proposition 2.5 of \cite{Mhn1997} for the second equality. 

\medskip

We now give the upper bound on $e(M)-\lambda(F/M)$. 

\begin{Theorem}\label{upper1}
Let $M$ be a finitely generated torsion-free $R$-module of rank $r$ with $I=I(M)$. 
Then $e(M)- \lambda (F/M) \leq 
\lambda (R/adj(I))$. Moreover, equality holds if and only if 
$M$ is integrally closed. 
\end{Theorem}

\begin{proof}
We may assume that $M$ has no free direct summands. 
Take a minimal reduction $N$ of $M$ and  
extend the minimal generating set of $N$ to one of the integral closure $\bar M$.  
Choosing a presenting matrix $A$ of $\bar M$ and its submatrix $B$ 
as in Lemma \ref{keylem},  
we have 
\begin{align*}
    e(M)-\lambda(F/M) &\leq e(\bar M)-\lambda(F/\bar M) \\
    &= \lambda(F/N)-\lambda(F/\bar M) \quad \text{by Theorem $\ref{br-formula}$}\\
    &= \lambda (\bar M/N) \\
    &=  \lambda (R/I_{n-r-1}(B)) \quad \text{by Lemma $\ref{keylem}$}\\
    &=  \lambda (R/adj(I(\bar M))) \quad \text{by Theorem $\ref{adj}$} \\
    &=  \lambda (R/adj(\bar{I})) \quad \text{by Theorem $\ref{kdy}$} \\ 
    &=  \lambda (R/adj(I)) \quad \text{by $(\ref{adjprop})$}. 
\end{align*}
It is clear that  equality holds if and only if $\lambda(F/M)=\lambda(F/\bar M)$ if and only if
$M=\bar M$. 
\end{proof}

As a consequence, we get a formula: 
$$e(M)-\lambda(F/M)=\lambda(R/adj(I))$$
for any integrally closed $R$-module $M$ with $I=I(M)$. 
This can be viewed as a natural extension of the well-known formula:  
$$e(I )-\lambda(R/I )=\lambda(R/adj(I ))$$
for any integrally closed $\fkm$-primary ideal $I $ in $R$ - see Proposition 3.3 of \cite{Lpm1994}. 
Thanks to this extension, 
we can readily get the following interesting formula 
\begin{equation}\label{lenmultid}
e(I)-e(M)=\lambda(R/I)-\lambda(F/M)
\end{equation}
proved in Corollary 4.3 of \cite{KdyMhn2015}. 
We will refer to this in the sequel as the {\it length-multiplicity identity}. 
In fact, we further prove the following: 

\begin{Corollary}\label{kdymhn}
Let $M$ be a finitely generated torsion-free $R$-module of rank $r$ with $I=I(M)$. Then $e(I)-e(M)\geq \lambda(R/I)-\lambda(F/M)$ holds. Moreover, equality holds if $M$ integrally closed. 
\end{Corollary}

\begin{proof}
First, we prove the equality $e(I)-e(M)=\lambda(R/I)-\lambda(F/M)$ when $M$ is integrally closed.
Note that, since $M$ is integrally closed, $I=I(M)$ is also integrally closed by Theorem $\ref{kdy}$. 
Applying Theorem \ref{upper1} to both $M$ and $I$, it follows that 
$e(M)-  \lambda (F/M) =  \lambda (R/adj(I))=
e(I)-  \lambda (R/I). $

Next, we prove the general inequality $e(I)-e(M) \geq \lambda(R/I)-\lambda(F/M)$. 
We first reduce to the case that $M \subset \mathfrak{m}F$. 
Write $M=M'\oplus G$ where $G$ is free and $M'$ has no free direct summand. Then the double $R$-dual $F=F'\oplus G$ which contains $M=M' \oplus G$ canonically. Since $F/M \cong F'/M'$, it follows that $I(M)=I(M')$, $e(M)=e(M')$, and $\lambda(F/M)=\lambda(F'/M')$. Hence, we assume that $M \subset \mathfrak{m} F$. 

Choose a minimal reduction $N$ and a matrix $\tilde{N}$ whose columns generate $N$ as in Proposition \ref{rk3prop}, i.e., if
$\widetilde{N}^T$ 
resolves $I(N)=(a_1,a_2,\dots , a_{r+1})$, then $(a_1, a_2)$ is a minimal reduction of $I(N)$. 
As in the proof of Proposition \ref{lower2}, applying Lemma \ref{keylem} when $M=I(N)$ and $N=(a_1, a_2)$, 
we have $e(I)-e(M)=\lambda(R/I_{r-1}(B))$ where $B$ is the submatrix of 
$\widetilde{N}^T$ 
obtained by deleting its first two rows. 
Let $B_j$ be the submatrix of $B$ obtained by deleting its $j$th column. Consider the following $R$-linear map  
\begin{equation*}
\begin{CD}
\phi: F @>>> \frac{I_{r-1}(B)+I(M)}{I(M)}
\end{CD}
\end{equation*}
defined by taking the basis vector $e_j$ of $F$ to the image of $\Delta_j:=(-1)^{j-1} det B_j$. 

\medskip

\noindent
{\bf{Claim:}} $M \subset \mathrm{Ker} \phi$. 

\medskip

Take any element $v=[v_1 \cdots v_r]^T \in M$ - regarded as an element of $F=R^r$. 
Form the matrix (say $C$) by concatenating the row vector $v^T$ with $B$. 
By definition of $I(M)$, $det (C) \in I(M)$. However, we can calculate $det (C)$ by expanding along the $v^T$ row. 
This gives $det (C) =v_1\Delta_1+\dots +v_r\Delta_r$. Thus, $\phi(v)=0$. 

We now have a surjective $R$-linear map 
\begin{equation*}
\begin{CD}
\phi': \frac{F}{M} @>>> \frac{I_{r-1}(B)+I(M)}{I(M)}. 
\end{CD}
\end{equation*}
This gives $\lambda(F/M) \geq \lambda(\frac{I_{r-1}(B)+I(M)}{I(M)})=\lambda(\frac{R}{I(M)})-\lambda(\frac{R}{I_{r-1}(B)+I(M)})$, 
and therefore we have 
$e(I)-e(M)=\lambda(\frac{R}{I_{r-1}(B)})
\geq \lambda(\frac{R}{I_{r-1}(B)+I(M)})
\geq \lambda(\frac{R}{I(M)})-\lambda(F/M). 
$
\end{proof}

\begin{Remark} 
There is a class of torsion-free modules $M$ that are not integrally closed such that the equality $e(I)-e(M)=\lambda(R/I)-\lambda(F/M)$ holds - see Example \ref{ex}. Thus the converse of Corollary \ref{kdymhn} does not hold. 
\end{Remark}

We next give the upper bound on $e(I(M))-e(M)$.

\begin{Corollary}\label{upper2}
Let $M$ be a finitely generated torsion-free $R$-module with $I=I(M)$. 
Then $e(I)-e(M) \leq \lambda(R/adj(I))$. Moreover, 
equality holds if and only if $I$ is integrally closed with 
$\lambda (R/I)=e(M)$.
\end{Corollary}

\begin{proof}
We may assume that $M$ has no free direct summands. 
Let $N$ be a minimal reduction of $M$.  Then
\begin{align*}
    e(I)-e(M) &= e(\bar{I})-  \lambda (R/I(N)) \quad \text{by Theorem \ref{br-formula} } \\
    &\leq e(\bar I)-  \lambda (R/\bar{I}) \quad \text{since $I(N) \subset \bar{I}$}\\
    &= e(I(\bar M))-  \lambda (R/I(\bar M)) \quad \text{by Theorem $\ref{kdy}$}\\
    &= e(\bar M)-  \lambda (F/\bar M) \quad \text{by Corollary \ref{kdymhn}}\\
    &= \lambda (R/adj(I(\bar M)) \quad \text{by Theorem \ref{upper1}}\\
    &=\lambda(R/adj(\bar I)) \quad \text{by Theorem $\ref{kdy}$} \\
    &=\lambda(R/adj(I)) \quad \text{by $(\ref{adjprop})$}.  
\end{align*}
The equality holds if and only if $\bar{I}=I(N)$ if and only if 
$I$ is integrally closed with $\lambda(R/I)=\lambda(R/I(N))=e(M)$. 
\end{proof}

\begin{proof}[Proof of Theorem \ref{main}] This is an immediate consequence of Theorem \ref{upper1} and Corollary~\ref{upper2}. 
\end{proof}

Next  we note that there is a large class of modules $K$ satisfying the equality $e(I(K))-e(K)=\lambda(R/adj(I(K)))$ in 
Corollary \ref{upper2}. 

\begin{Corollary}\label{eqex}
Let $M$ be an integrally closed module of rank $r$ with a minimal free resolution: 
\begin{equation*}
  \begin{CD}
     0 @>>> R^{n-r}  @>{\widetilde{K}^T }>>  R^{n}
     @>>> M  @>>>  0.  
  \end{CD}
\end{equation*}
Let $K$ be the image of the map defined by the transpose of the presenting matrix $\widetilde{K}^T$.  
Then the module $K$ satisfies the equality:  
$$e(K)=e(I(K))-\lambda(R/adj(I(K))). $$
\end{Corollary}

\begin{proof}
Let $n=\mu(M)$ and $I=I(M)$. 
Note that the rank of $K$ is $n-r$. 
Then $I(K)=I_{n-r}(K)=I$ and $I_{n-r-1}(K)=adj(I)$. 
The first equalities follow from the proof of Proposition 2.2 of \cite{Kdy1995}, 
and the second one follows from Theorem \ref{adj}. 
By Proposition \ref{lower2}, 
$$e(I)-e(K) \geq \lambda(R/I_{n-r-1}(K))=\lambda(R/adj(I)). $$
The converse $e(I)-e(K) \leq \lambda(R/adj(I))$ also holds by Corollary \ref{upper2}. 
Thus, we have the equality $e(K)=e(I)-\lambda(R/adj(I))$. 
\end{proof}

Our next result is another consequence of Theorem \ref{upper1}, 
which is a result towards the classification of integrally closed modules. 
To state the result, we recall the fact that for any finitely generated torsion-free $R$-module $M$, 
the inequality $$\mu(M) \leq ord(I(M))+rk(M)$$ always holds, and the equality $\mu(M) = ord(I(M))+rk(M)$ holds if and only if $M$ is contracted - see Propositions 2.2 and 2.5 of \cite{Kdy1995}. In particular, integrally closed modules satisfy the equality $\mu(M) = ord(I(M))+rk(M)$ - see Proposition~4.3 of \cite{Kdy1995}. 

For a given integrally closed $\mathfrak m$-primary ideal $I$ with  $ord(I)=r$, we consider 
two collections of torsion-free $R$-modules $\mathcal M_I$ and $\mathcal K_I$ associated to $I$.
Define
\begin{align*}
\mathcal M_I&=\{M \mid M \ \text{is integrally closed of rank $r$,} \ I(M)=I, \ M \subset \fkm M^{\ast \ast}\},\\
\mathcal K_I&=\{K \mid K \ \text{is contracted of rank $r$}, I(K)=I, I_{r-1}(K)=adj(I), \ K \subset \fkm K^{\ast \ast} \}.
\end{align*}
Here we set $I_0(K)=R$.
Then we have the following theorem. 

\begin{Theorem}\label{1to1}
Let $I$ be an integrally closed $\mathfrak m$-primary ideal with $ord(I)=r$. Then 
there is a bijection
\begin{equation*}
  \begin{CD}
     \psi: \mathcal M_I/{\cong} @>>> \mathcal K_I/{\cong}
  \end{CD}
\end{equation*}
between the isomorphism classes of elements of $\mathcal M_I$ and  those of $\mathcal K_I$.
\end{Theorem}

\begin{proof}
Let $M \in \mathcal M_I$. Then $\mu(M)=ord(I(M))+rk(M)=2r$ since $M$ is integrally closed with $I(M)=I$. 
We consider a matrix $\widetilde{K}$ such that its transpose $\widetilde{K}^T$ is a minimal presenting matrix of $M$ and let $K$ be the module generated by its columns. Thus we have an exact sequence,   
\begin{equation*}
  \begin{CD}
   0 @>>> R^r @>{ \widetilde{K}^T }>> R^{2r} @>>> M @>>> 0. 
  \end{CD}
\end{equation*}
Then $K$ is of rank $r$ with $K \subset \fkm K^{\ast \ast}$, and we have $I(K)=I(M)=I$ as in the proof of Corollary \ref{eqex}. 
Since $\mu(K)=2r=ord(I(K))+rk(K)$, $K$ is contracted. By Theorem \ref{adj}, $I_{r-1}(K)=adj(I)$. Hence, $K \in \mathcal K_I$. It is easy to see that the isomorphism class of $K$ depends only on that of $M$. 
Let 
\begin{equation*}
  \begin{CD}
     \psi: \mathcal M_I/{\cong} @>>> \mathcal K_I/{\cong}
  \end{CD}
\end{equation*}
be the map defined by 
$\psi([M])=[K]$.

Similarly, for any $K \in \mathcal K_I$, consider a matrix $\widetilde{M}$ such that its transpose $\widetilde{M}^T$ is a minimal presenting matrix of $K$ and let $M$ be the module generated by its columns, so that we have an exact sequence, 
\begin{equation*}
  \begin{CD}
   0 @>>> R^r @>{\widetilde{M}^T}>> R^{2r} @>>> K @>>> 0. 
  \end{CD}
\end{equation*}
Then $M$ is of rank $r$ with $M \subset \fkm M^{\ast \ast}$ and $I(M)=I(K)=I$. Moreover, we claim that 
the module $M$ is integrally closed. 
Indeed, by Proposition \ref{lower1} and Theorem \ref{upper1}, we have the following inequalities: 
\begin{align*}
&e(M)-\lambda(M^{\ast \ast}/M) \geq \lambda(R/I_{r-1}(K)) \\
& e(M)-\lambda(M^{\ast \ast}/M) \leq \lambda(R/adj(I)).  
\end{align*}
Since $I_{r-1}(K)=adj(I)$, we have the equality 
$$e(M)-\lambda(M^{\ast \ast}/M)=\lambda(R/adj(I)). $$ 
Hence, by Theorem \ref{upper1}, $M$ is integrally closed and therefore, $M \in \mathcal M_I$.  
Again the isomorphism class of $M$ depends only on that of 
$K$ and we have a map
\begin{equation*}
  \begin{CD}
     \psi': \mathcal K_I/{\cong} @>>> \mathcal M_I/{\cong} 
  \end{CD}
\end{equation*}
defined by 
$\psi'([K])=[M]$. 

It is clear by construction that $\psi \circ \psi'=id$ and $\psi' \circ \psi=id$, and hence define mutually inverse bijections. \end{proof}

\section{Remarks} 

Let $I $ and $J $ be $\fkm$-primary ideals in $R$. 
Then for large integers $p$ and $q$, the length function
$\lambda (R/I ^pJ ^q)$ can be written
in the form:
\begin{equation}\label{hilb}
\lambda (R/I ^pJ ^q)=e(I )\binom{p}{2}+
e_1(I | J )pq+e(J )\binom{q}{2}+(\text{lower terms})
\end{equation}
for some positive integer $e_1(I |J )$ called 
the {\it mixed multiplicity} of $I $ and $J $. 
If $I $ and $J $ are integrally closed, then the equality 
\begin{equation}\label{mixed}
    e_1(I |J )=  \lambda (R/I J )-  \lambda (R/I )-  \lambda (R/J ) 
\end{equation}
holds - see Corollary 3.7 of \cite{Lpm1988}. We call it the {\it mixed multiplicity formula}.  
The length-multiplicity identity 
 (\ref{lenmultid}) 
for integrally closed modules can 
be viewed as an analogue of 
this formula. 

Let $I $ and $J $ be  $\fkm$-primary ideals, and let 
$M=I \oplus J $. 
Then, since $I(M)=I J $, considering the function (\ref{hilb}) with $p=q$, we have
$$e(I(M))=e(I )+2e_1(I |J )+e(J ). $$
On the other hand, the Buchsbaum-Rim multiplicity $e(M)$ can be expressed as 
$$e(M)=e(I )+e_1(I |J )+e(J ) $$
by a result of Kirby and Rees - see Proposition 4.1 of \cite{KrbRes1996}. 
Hence, the difference $e(I(M))-e(M)$ is just the mixed multiplicity $e_1(I | J)$ in this case. 
Therefore, if we further assume that $I$ and $J$ are integrally closed or equivalently if $M=I \oplus J$ is integrally closed, then
\begin{align*}
e_1(I | J )&=e(I(M))-e(M) \\
&=\lambda(R/I(M))-\lambda(F/M) \quad \text{by Corollary $\ref{kdymhn}$}\\
&=\lambda(R/I J )-\lambda(R/I )-\lambda(R/J ). 
\end{align*}
Thus, the mixed multiplicity formula (\ref{mixed}) follows from the length-multiplicity identity
 (\ref{lenmultid}) 
in a special case. 

In general, it is known that
for any $\fkm$-primary (not necessarily integrally closed) ideals $I $ and $J $, 
the inequality
\begin{equation}\label{mixedineq}
e_1(I |J ) \geq  \lambda (R/I J )-  \lambda (R/I )-  \lambda (R/J ) 
\end{equation}
holds true, and equality holds if and only if there exists $a \in I $ and $b \in J $ such that 
$I J =aJ +bI $ - see results and the arguments of \cite{Vrm1990}. 
The inequality in Corollary \ref{kdymhn} can be viewed as a generalisation of (\ref{mixedineq}). 

We conclude with an example illustrating our results 
obtained in Corollary \ref{kdymhn}. 

\begin{Example}\label{ex}{\rm 
Let $x, y$ be a regular system of parameters for $R$, and let $a, b, c \geq 1$ be integers 
with $1 \leq a \leq c < b \leq a+c$. Consider a submodule $M(a,b,c)$ of $F=R^2$ 
generated by columns of the following matrix: 
$$
\begin{bmatrix}
y^a&x^b&0&x^cy^c\\
x^a&0&y^b&0
\end{bmatrix}. 
$$
Let $M=M(a, b, c)$ and $I=I(M)$. Then the following inequality 
$$e(I)-\lambda(R/I) \geq e(M)-\lambda(F/M)$$
holds true. Moreover, the equality $e(I)-\lambda(R/I) = e(M)-\lambda(F/M)$ holds if and only if $a=b-c$. 
}
\end{Example}

\begin{proof}
Let $J=(x^{a+b}, y^{a+b})$. Then $J$ is a minimal reduction of $I$ since $J \subset I \subset \bar{J}$. 
Let $N$ be a submodule of $M$ generated by the first $3$ generators of $M$:
$$
\begin{bmatrix}
y^a&x^b&0\\
x^a&0&y^b
\end{bmatrix}. 
$$
Since 
$I(N)=(x^{a+b}, x^by^b, y^{a+b})$ is a reduction of $I$, $N$ is a minimal reduction of $M$. 
Hence, $e(I)=\lambda(R/J)=(a+b)^2$ and $\lambda(R/I)=(a+b)^2-a^2-(b-c)^2$. Thus, 
$$e(I)-\lambda(R/I)=a^2+(b-c)^2. $$
A presenting matrix of $M$ in Lemma \ref{keylem} is 
$$A=\begin{bmatrix}
0&x^cy^b \\
-y^c&0\\
0&-x^{a+c}\\
x^{b-c}&-y^{a+b-c}
\end{bmatrix}, 
$$
and the submatrix $B$ is $\begin{bmatrix} x^{b-c} & -y^{a+b-c} \end{bmatrix}. $
Therefore, 
\begin{align*}
e(M)-\lambda(F/M)&=\lambda(M/N)\\
&=\lambda(R/I_1(B)) \quad \text{by Lemma \ref{keylem}} \\
&=\lambda(R/(x^{b-c}, y^{a+b-c})) \\
&=a(b-c)+(b-c)^2. 
\end{align*}
Thus, we have the assertions. 
\end{proof}

In this example, if we further assume that $a+b \leq 2c$, then a little thought shows that the equality $M^2=NM$ holds in $\Sym^2_R(F)=R^3$. 
Therefore, if we take $a=2, b=4, c=3$, then the module $M=M(2, 4, 3)$ satisfies 
$M^2=NM$, but it does not satisfy the length-multiplicity identity
 (\ref{lenmultid}). 
Hence, the equality $M^2=NM$ is not a sufficient condition for the length-multiplicity identity. 

This cannot occur in the case where $M$ is a direct sum of two $\fkm$-primary ideals. 
In fact, when $M=I \oplus J $, one can show that the equality $M^2=NM$ holds for some minimal reduction $N$ of $M$
 if and only if there exist $a \in I$ and $b \in J$ such that $IJ=aJ+bI$ and both of the ideals $I, J$ have reduction number one.  
Therefore, the equality $M^2=NM$ implies the length-multiplicity identity when $M=I\oplus J$.
Thus, the module $M(2,4,3)$ is not a direct sum of ideals, hence, it is indecomposable. 

The equality $M^2=NM$ is also not a necessary condition for the length-multiplicity identity even in the case where $M=I\oplus J$. 
Indeed, if we take $I=(x,y^6)$ and $J=(x^3,xy^4, y^6)$, then $IJ=xJ+y^6J$ but $J^2 \neq (x^3,y^6)J$. Hence, $M=I\oplus J$ satisfies the length-multiplicity identity but 
$M^2\neq NM$ for any minimal reduction $N$ of $M$. 
This raises the natural question stated below.

\begin{Question} What is a characterisation of  torsion-free modules $M$ that satisfy the length-multiplicity identity?
\end{Question}

\section*{Acknowledgments}
The authors would like to thank the referee for the helpful comments. 
The first named author was supported by JSPS KAKENHI Grant Number JP23K03054.


\end{document}